\documentclass[12pt]{article}
\usepackage{amsmath, amsthm, amsfonts}
\usepackage[T2A]{fontenc}
\usepackage[utf8]{inputenc}
\usepackage{ragged2e}
\usepackage{mathtools}
\usepackage{ulem}
\usepackage{lipsum}
\usepackage[top=1in, bottom=1in, left=1in, right=1in]{geometry}
\usepackage{microtype}
\usepackage{amssymb}
\usepackage[hidelinks]{hyperref}
\usepackage{listings}

\usepackage{amsthm}

\newtheorem{theorem}{Theorem}[section]
\newtheorem{lemma}[theorem]{Lemma}

\newtheorem{corollary}[theorem]{Corollary}
\newtheorem{definition}[theorem]{Definition}

\begin{document}

\vspace{100pt}

\begin{center}
    \textbf{\LARGE Gaps, isolated points, and the origin of the ray in the one-dimensional Dirichlet spectrum}
\end{center}

\begin{center}
    \textbf{\Large Maksim Lavrov}\\[4pt]
    Faculty of Mechanics and Mathematics,\\
    Lomonosov Moscow State University
\end{center}

\begin{abstract}
    
We study the one-dimensional Dirichlet spectrum and improve several classical results of Ivanov. First, we obtain a new upper bound for the origin \(\mu\) of the ray contained in the spectrum,
\[
\mu\leq \frac{448875+1611\sqrt{67773}}{46982}
       =18.4809\ldots.
\]
We also exhibit an explicit gap \((L,R)\) in the Dirichlet spectrum, which yields the improved lower bound
\[
\mu\geq R=5.9596\ldots.
\]
The proof of the latter result includes an exact computer-assisted finite search over admissible continued-fraction words. We further prove that all points \(Z_k\), \(k\geq2\), introduced by Ivanov are isolated points of the Dirichlet spectrum. We determine the asymptotic behavior of the neighboring gaps around \(Z_k\) and describe the equivalence classes of all irrational numbers whose Dirichlet value is \(Z_k\).

\end{abstract}

\section{Introduction}

Let $\alpha\in\mathbb{R}\setminus\mathbb{Q}$, and let $[a_0;a_1,a_2,\dots]$ - be its continued fraction expansion. We define
\[
d_n(\alpha):=[a_n;a_{n-1},\dots,a_1][a_{n+1};a_{n+2},\dots],
\]
\[
d(\alpha):=\limsup_{n \to \infty} d_n(\alpha).
\]
The set 
\[
\mathcal{D}:=\{d(\alpha)\ |\ \alpha\in\mathbb{R}\setminus\mathbb{Q}\}
\]
is called the Dirichlet spectrum.

Many of its properties were discovered by Ivanov (see \cite{Ivanov1982}, \cite{Ivanov1978}) and have seen little to no improvement since the 1980's. In the present paper we build upon his work and improve his results.

\section{Main results}

\subsection{Upper bound of the origin of the ray}

In article \cite{DivisNovak1971} Divis and Novak prove the existence of a number $R$, such that 
\[
[R,+\infty]\subset\mathcal{D}.
\]
Let
\[
\mu:=\inf\{R\ | \ [R,+\infty]\subset\mathcal{D}\}.
\]
Divis and Novak achieved the following upper bound
\[
\mu\leq12+8\sqrt{2}=23.3136\dots
\]
Later Ivanov improved their result. In \cite{Ivanov1982} he showed that
\[
\mu\leq10+6\sqrt{2}=18.4852\dots
\]
For decades this stayed as the best upper bound. Recently in \cite{PengWangWang2026} it was proven that
\[
\mu\leq\frac{19425+111\sqrt{26565}}{2030}=18.4811\dots
\]
One of our goals is to improve this result by proving the following theorem.
\begin{theorem}\label{ray upper bound}
    \begin{equation}
        \mu\leq\frac{448875+1611\sqrt{67773}}{46982}=18.4809\dots
    \end{equation}
\end{theorem}

\subsection{Lower bound of the origin of the ray}

In \cite{Ivanov1982} Ivanov proved that
\[
\mu\geq\frac{5+3\sqrt{5}}{2}=5.8541\dots
\]
The goal of this section will be to improve this bound. In fact, not only will we improve it, we will also provide a new empty interval in the Dirichlet spectrum. Define
\[
L:=7-\frac{8\sqrt{21}}{35}=5.9525\dots
\]
\[
R:=\frac{381780528649+35642709\sqrt{21}}{64088662090}=5.9596\dots
\]
We will prove the following theorem.
\begin{theorem}\label{ray lower bound}
    \[
    \mathcal{D}\cap(L,R)=\emptyset.
    \]
    In particular,
    \[
    \mu\geq R.
    \]
\end{theorem}

\subsection{Isolated points in Dirichlet spectrum}

In ~\cite{Ivanov1982} Ivanov introduces the set $M(a), a\in\mathbb{N}$, which is the set consisting of all irrational numbers $x=[x_0;x_1,x_2,\dots]$, such that
\[
\limsup_{n\to\infty}x_n=a.
\]
And its subset $M(a,b), a\geq b\geq1$, such that the block $a,b$ or $b,a$ appears infinitely many times in the sequence $\{x_n\}$, while the blocks $a,c$ or $c,a$ where $c\geq b+1$ appear only finitely many times. He then introduces certain subsets $M_{k,\omega}$ of $M(3,1)$, where $k\in\mathbb{N}$, and $\omega\in\{0,1\}$. $M_{k,1}$ is defined for $k\geq 1$, whereas $M_{k,0}$ is defined for $k\geq 2$. Let
\[
\mathfrak{m}_{k,\omega}:=\{d(x)\ | \ x\in M_{k,\omega}\}.
\]
And
\[
\tilde{\mathfrak{m}}_{k,0}:=\mathfrak{m}_{k,0}\setminus \{\min \mathfrak{m}_{k,0}\}.
\]

Let $X,Y$ be two subsets of real numbers. By $X\prec Y$ we denote the fact that
\[
\sup X< \inf Y.
\]

Ivanov proves the following theorem.
\begin{theorem}\label{Ivanov}
    \begin{enumerate}
        \item $\mathcal{D}\ \cap\ \left(3+\sqrt{5}, \frac{5+3\sqrt{5}}{2}\right)=\mathfrak{m}_{1,1}\cup\bigcup_{k=2}^\infty (\mathfrak{m}_{k,0}\cup \mathfrak{m}_{k,1})$
        \item $\mathfrak{m}_{1,1}\prec\mathfrak{m}_{2,0}\prec\mathfrak{m}_{2,1}\prec\tilde{\mathfrak{m}}_{3,0}\prec\mathfrak{m}_{3,1}\prec\tilde{\mathfrak{m}}_{4,0}\prec\mathfrak{m}_{4,1}\prec\dots\prec\mathfrak{m}_{k,1}\prec\tilde{\mathfrak{m}}_{k+1,0}\prec\mathfrak{m}_{k+1,1}\prec\dots$
    \end{enumerate}
\end{theorem}

Let
\begin{equation}
    Z_k:=\min \mathfrak{m}_{k,0}, \ k\geq 2,
\end{equation}
\begin{equation}
    Q_k:=\min \tilde{\mathfrak{m}}_{k,0}, k\geq 2,
\end{equation}
\begin{equation}
    P_k:=\max \mathfrak{m}_{k-1,1}, \ k\geq 3.
\end{equation}
Ivanov gives explicit formulas for these values.
\begin{equation}\label{Zk}
    Z_k=3[\overline{(1)_{2k-1},3}]+1,
\end{equation}
\begin{equation}\label{Qk}
    Q_k=3[(1)_{2k-1},2,\overline{1,3}]+1,
\end{equation}
\begin{equation}\label{Pk}
    P_k=[3,(1)_{2k-5},2,\overline{1,3}][(1)_{2k-3},3,(1)_{2k-3},2,1,1,2,\overline{1,3}].
\end{equation}

Moreover, he shows that the points $Z_2,Z_3$ are isolated in $\mathcal{D}$. The goal of this section is to prove the following theorem.
\begin{theorem}\label{isolated points}
    For all $k\geq 2$ the points $Z_k$ are isolated in $\mathcal{D}$.
\end{theorem}
Using theorem ~\ref{Ivanov} and beforementioned results of Ivanov it is sufficient to show that $P_k<Z_k$ for all $k\geq 4$.

Moreover, we will show how the intervals $Z_k-P_k$, $Q_k-Z_k$, $Q_k-P_k$ shrink as $k$ approaches infinity. We will prove the following result.
\begin{theorem}\label{asymptotics}
    Let $\{F_n\}$ be the sequence of Fibonacci numbers, where $F_0=0, F_1=1$. The following holds
    \begin{equation}\label{Zk-Pk}
        Z_k-P_k=\frac{C_-}{F_{2k-3}^2}+O(F_{2k-3}^{-4}),
    \end{equation}
    \begin{equation}\label{Qk-Zk}
        Q_k-Z_k=\frac{C_+}{F_{2k-3}^2}+O(F_{2k-3}^{-4}),
    \end{equation}
    \begin{equation}\label{Qk-Pk}
        Q_k-P_k=\frac{C}{F_{2k-3}^2}+O(F_{2k-3}^{-4}).
    \end{equation}
    Where
    \[
    C_-=\frac{-301+126\sqrt{5}+2\sqrt{21}+\sqrt{105}}{8},
    \]
    \[
    C_+=\frac{3(247+7\sqrt{21}-111\sqrt{5}-3\sqrt{105})}{16},
    \]
    \[
    C=\frac{139+25\sqrt{21}-81\sqrt{5}-7\sqrt{105}}{16}.
    \]
\end{theorem}

Finally we will consider the number of equivalence classes of each point $Z_k$. Two numbers $x=[x_0;x_1,x_2,\dots]$ and $y=[y_0;y_1,y_2,\dots]$ are called equivalent, if for some integers $s,l$
\[
x_{s+m}=y_{l+m} \ \forall \ m\in\mathbb{N}.
\]
Clearly if $x$ and $y$ are equivalent then $d(x)=d(y)$.

In \cite{Ivanov1982} Ivanov states that the point $Z_3=3+2\sqrt{2}$ is interesting, since it is the first example of a point in the Dirichlet spectrum which contains more than one preimage, and the number of its preimages is not continual. This, however, is not the case for other points $Z_k$. This fact will be proven in the following theorem.
\begin{theorem}\label{equivalence classes}
    For the isolated points $Z_k$ the following holds
    \begin{equation}\label{equivalence classes Z3}
        d^{-1}(Z_3)=\{x:x\sim[\overline{2}]\}\cup \{x:x\sim[\overline{(1)_5,3}]\}\cup\{x:x\sim[\overline{4,1}]\},
    \end{equation}
    \begin{equation}\label{equivalence classes Zk}
        d^{-1}(Z_k)=\{x:x\sim[\overline{(1)_{2k-1},3}]\}, \ k\neq 3.
    \end{equation}
\end{theorem}

\section{Upper bound of the origin of the ray}

\subsection{The plan}

In \cite{Hall1947} Hall proved the following theorem

\begin{theorem}[Hall]\label{Hall's theorem}
Let $C(A)\subset A$ and $C(B)\subset B$ be Cantor sets constructed
on closed intervals $A$ and $B$, respectively. Suppose that, at every
stage of the constructions, each deleted open interval is no longer
than either of the two adjacent intervals that remain. If
\[
    \frac{1}{3}\leq \frac{|A|}{|B|}\leq 3,
\]
then
\[
    C(A)+C(B)=A+B.
\]
\end{theorem}

This fact will be used later.

\begin{definition}
Set
\[
F_1=
\left\{
[4;3,a_2,a_3,\ldots]:
\begin{array}{c}
a_j\in\{1,2,3,4\}\quad (j\ge 2),\\[2mm]
44,\ 41414,\ 4131414,\ 4141314\\
\text{do not occur in }(a_0,a_1,a_2,\ldots)
\end{array}
\right\}.
\]
\end{definition}

From now on we will call the strings $44, 41414, 4131414, 4141314$ forbidden words. Using properties of continued fractions it is easy to show that

\begin{equation}
    L_F:=\min F_1=[4;3,\overline{1,4,1,4,1,3,1,3}], \ U_F:=\max F_1 = [4;3,\overline{4,1,4,1,3,1,3,1}] 
\end{equation}

First we will show that $F_1$ is a Cantor set. Then, using Hall's theorem, we will prove the fact that 

\[
F_1*F_1=[L_F^2,U_F^2]
\]

Afterwards, we will introduce the constant

\[
C:=\frac{448875+1611\sqrt{67773}}{46982}.
\]

It is easy to show that

\[
L_F^2<C<10+6\sqrt{2}<U_F^2
\]

Thus, every number

\[
\alpha\in[C,10+6\sqrt{2}]
\]

can be represented as 

\[
\alpha=XY,\ X=[x_0;x_1,x_2,\dots]\in F_1,\ Y=[y_0;y_1,y_2,\dots]\in F_1.
\]

Let

\[
S_i:=(x_i,x_{i-1},\dots,x_1,x_0,y_0,y_1,\dots,y_{i-1},y_i),
\]

and

\[
\xi:=[S_1,2,S_2,2,S_3,2,\dots].
\]

For $\xi$ we will show that

\[
d(\xi)=\alpha.
\]

Therefore, 

\[
[C,10+6\sqrt{2}]\subset \mathcal{D}.
\]

Since $\mu\leq10+6\sqrt{2}$ (see \cite{Ivanov1982}), this will conclude our proof.

\subsection{$F_1$ as a Cantor set}

In this section we demonstrate how $F_1$ can be constructed as a Cantor set.

\begin{definition}
    A finite string $w=(4,3,a_2,\dots,a_n)$ is called an admissible prefix if it consists only of digits $1,2,3,4$ and does not contain any forbidden words.

    By $\mathcal{K}(w)$ we denote the set 
    \[
    \{[4;3,a_2,\dots,a_n,a_{n+1},\dots]\in F_1\}.
    \]
    
    And by $T[w]$ we denote the closed interval
    \[
    [\min \mathcal{K}(w),\max \mathcal{K}(w)].
    \]

    Any number of the form
    \[
    [4;3,a_2,\dots,a_n,a_{n+1},\dots]\in F_1
    \]
    will be called an admissible continuation of $w$. The number $[a_{n+1};a_{n+2},\dots]$ will be called the tail of the admissible continuation.
\end{definition}

We will need the following lemma.

\begin{lemma}\label{lem:intersection}
    Let $w=(a_0,a_1,\dots,a_n)$ be an admissible prefix, such that for some digit $d$ the prefixes $wd$ and $w(d+1)$ are admissible as well. Then the closed intervals $T[wd]$ and $T[w(d+1)]$ do not intersect. Moreover, there is an open interval between them.
\end{lemma}

\begin{proof}
    Let
    \[
    \Phi_w(t):=[a_0;a_1,\dots,a_n,t].
    \]
    This function is continuous and strictly monotonic.

    For any tail of the form $t=[d;b_1,b_2,\dots]=d+\frac{1}{[b_1;b_2,\dots]}$ the following inequalities hold
    \[
    d+\frac{1}{5}<t<d+1.
    \]

    Similarly, for any tail $t$ of the form $[d+1;c_1,c_2,\dots]$ the following holds
    \[
    d+1+\frac{1}{5}<t<d+2.
    \]

    Therefore, since 
    \[
    T[wd]\subset[\Phi_w(d+\frac{1}{5}),\Phi_w(d+1)], \\
    T[w(d+1)]\subset[\Phi_w(d+1+\frac{1}{5}),\Phi_w(d+2)],
    \]
    and the function $\Phi_w$ is strictly monotonic and continuous, the sets $T[wd]$ and $T[w(d+1)]$ do not intersect and there is an interval between them.
\end{proof}

\begin{definition}
    For any admissible prefix $w$ we will call the longest beginning of a forbidden word at its end its suffix. We will denote its suffix as $\sigma(w)$.
\end{definition}

The full list of suffixes is
\[
S:=\{\emptyset,4, 41, 413, 414, 4131, 4141, 41314, 41413, 413141, 414131 \}.
\]

For example, for the admissible prefix $w=(4,3,2,4,1,4)$ we have $\sigma(w)=414$. Whereas for the prefix $w=(4,3,2,2,1)$ we have $\sigma(w)=\emptyset$.

The following table shows which digits can be added to an admissible prefix with given suffix and which suffix the new prefix has. Be "-" we denote the digits that cannot be added to the end.

\[
\begin{array}{c|c|c|c|c|c}
\text{prefix }\sigma & 1 & 2 & 3 & 4 & \text{admissible digits} \\ \hline
\varnothing & \varnothing & \varnothing & \varnothing & 4 & 1,2,3,4 \\
4           & 41          & \varnothing & \varnothing & - & 1,2,3 \\
41          & \varnothing & \varnothing & 413         & 414 & 1,2,3,4 \\
413         & 4131        & \varnothing & \varnothing & 4 & 1,2,3,4 \\
414         & 4141        & \varnothing & \varnothing & - & 1,2,3 \\
4131        & \varnothing & \varnothing & \varnothing & 41314 & 1,2,3,4 \\
4141        & \varnothing & \varnothing & 41413       & - & 1,2,3 \\
41314       & 413141      & \varnothing & \varnothing & - & 1,2,3 \\
41413       & 414131      & \varnothing & \varnothing & 4 & 1,2,3,4 \\
413141      & \varnothing & \varnothing & 413         & - & 1,2,3 \\
414131      & \varnothing & \varnothing & \varnothing & - & 1,2,3
\end{array}
\]

From this table we can see that the set of admissible digits is always either $\{1,2,3\}$ or $\{1,2,3,4\}$.

\begin{definition}
    For a given prefix $\sigma$ we denote $D(\sigma)$ the digit $3$ or $4$, such that the admissible digits for $\sigma$ are $\{1,2,\dots,D(\sigma)\}$
\end{definition}

Now we will describe a Cantor set and prove that it is $F_1$.

We start with the closed interval $I_0:=T[4,3]$. Let $T[w]$, where $w$ is an admissible prefix, be a closed interval we have on some step. Let
\[
D:=D(\sigma(w)).
\]
Then $wr$, where $1\leq r\leq D$ are all admissible prefixes.

By $J(w,r)$ we denote the smallest closed interval containing $T[wr],T[w(r+1)],\dots,T[wD]$. By lemma \ref{lem:intersection} for $r<D$ the closed intervals $T[wr]$ and $J(w,r+1)$ do not intersect, and there is an interval between them. We denote this interval $G(w,r)$.

If $D=4:$
\begin{itemize}
    \item We delete $G(w,1)$ and have $T[w1]$ and $J(w,2)$ left.
    \item We delete $G(w,2)$ and have $T[w2]$ and $J(w,3)$ left.
    \item We delete $G(w,3)$ and have $T[w3]$ and $T[w4]$ left.
\end{itemize}

If $D=3:$
\begin{itemize}
    \item We delete $G(w,1)$ and have $T[w1]$ and $J(w,2)$ left.
    \item We delete $G(w,2)$ and have $T[w2]$ and $T[w3]$ left.
\end{itemize}

Eventually we get the following set
\[
F:
=
I_0
\setminus
\bigcup_{\substack{w\ \text{is an admissible prefix}\\[1mm] 1\le r < D(\sigma(w))}}
G(w,r).
\]

\begin{lemma}
    $F=F_1$.
\end{lemma}

\begin{proof}
    \begin{itemize}
        \item Let $x\in F_1$. Assume that $x$ got deleted at the step where we delete $G(w,r)$. This implies that $x\in J(w,r)$. This means that the first $n$ digits of $x$, where $n$ is some natural number form the admissible prefix $w$, but its next digit in at least $r$. However, the fact that $x\notin T[wr]\cup J(w,r+1)$ implies that the first $n+1$ digits of the continued fraction expansion of $x$ do not form an admissible prefix, which leads to a contradiction. Therefore, $x\in F$.
        
        \item Let $x\in F$. As we construct our Cantor set F we get an infinite sequence of closed intervals $I_n=T[4,3,a_2,a_3,\dots]$, such that $x\in I_n$ for all $n$. The sequence $4,3,a_2,a_3,\dots$ does not contain any forbidden words. Therefore, $y:=[4;3,a_2,a_3,\dots]\in F_1$. Since $|I_n|\to 0$ as $n\to \infty$ and $x\in I_n$, $x=y$. Thus, $x\in F_1$
    \end{itemize}
\end{proof}

\subsection{Thickness of $\ln F_1$}

\begin{definition}
    Let $I$ be a closed interval. Let $K\subset I$ be a Cantor set, and $G_1,G_2,\dots$ be the deleted intervals. For a fixed order of deleted intervals we can represent a closed interval from which $G_n$ is deleted as
    \[
    B_n^L\ \cup G_n \ \cup B_n^R.
    \]
    For this given order $\mathcal{P}$ we define
    \[
    \tau_{\mathcal P}(K)
    =
    \inf_n
    \min\left\{
    \frac{|B_n^L|}{|G_n|},
    \frac{|B_n^R|}{|G_n|}
    \right\}.
    \]
    And
    \[
    \tau(K):=\sup_\mathcal{P}\tau_{\mathcal{P}}(K).
    \]
\end{definition}

We will use ~\ref{Hall's theorem} to prove that
\begin{equation}
    \ln F_1+\ln F_1=[2\ln L_F,2\ln U_F].
\end{equation}

From this it follows that
\begin{equation}
    F_1*F_1=[L_F^2,U_F^2].
\end{equation}

To apply this theorem it is sufficient to show that
\begin{equation}
    \tau(\ln F_1)>1.
\end{equation}

Which is what we will do in this section.

\begin{definition}
    For a given prefix $\sigma$ we will denote $\mathcal{H}_\sigma$ the set of all tails of an admissible word whose prefix is $\sigma$.

    Let
    \[
    m_\sigma:=\min \mathcal{H}_\sigma, \ M_\sigma:=\max \mathcal{H}_\sigma.
    \]
    By $T(\sigma,d)$ we will denote the prefix of a word whose suffix is $\sigma$ after an admissible digit $d$ is added to its end.
\end{definition}

Clearly $m_\sigma$ and $M\sigma$ satisfy the following system of equations
\begin{equation}\label{system1}
    m_\sigma=1+\frac{1}{M_{T(\sigma,1)}},\ M_\sigma=D(\sigma)+\frac{1}{m_{T(\sigma,D(\sigma))}}.
\end{equation}

Let
\[
\rho:=\sqrt{67773}, \ S:=\sqrt{195}.
\]

Finding the positive solutions to system~(\ref{system1}) yields the following result.
\[
\begin{array}{c|c|c}
\sigma & m_\sigma & M_\sigma \\ \hline
\varnothing
& \dfrac{167+\rho}{354}
& \dfrac{3(187+\rho)}{278}
\\[2mm]
4
& \dfrac{551+3\rho}{1102}
& \dfrac{3(171+\rho)}{338}
\\[2mm]
41
& \dfrac{167+\rho}{354}
& \dfrac{551+3\rho}{278}
\\[2mm]
413
& \dfrac{9+S}{19}
& \dfrac{3(187+\rho)}{278}
\\[2mm]
414
& \dfrac{187+\rho}{354}
& \dfrac{3(171+\rho)}{338}
\\[2mm]
4131
& \dfrac{167+\rho}{354}
& \dfrac{10+S}{5}
\\[2mm]
4141
& \dfrac{167+\rho}{354}
& \dfrac{3(167+\rho)}{338}
\\[2mm]
41314
& \dfrac{10+S}{19}
& \dfrac{3(171+\rho)}{338}
\\[2mm]
41413
& \dfrac{171+\rho}{342}
& \dfrac{3(187+\rho)}{278}
\\[2mm]
413141
& \dfrac{167+\rho}{354}
& \dfrac{9+S}{6}
\\[2mm]
414131
& \dfrac{167+\rho}{354}
& \dfrac{3(171+\rho)}{338}
\end{array}
\]

For further calculations we will need the following result
\begin{equation}
    \frac{U_F}{L_F}<\frac{253}{250}.
\end{equation}

Let $w=(a_0,\dots,a_n)$ be an admissible word, $\sigma:=\sigma(w)$, $D:=D(w)$. Let $1\leq r<D$. We will consider a step in the construction of the Cantor set $F_1$, where we delete the interval $G(w,r)$ and have the closed intervals $T[wr]$ and $J(w,r+1)$ left. Let
\[
A=r+\frac{1}{M_{T(\sigma,r)}},
\qquad
B=r+\frac{1}{m_{T(\sigma,r)}},
\]
\[
C=r+1+\frac{1}{M_{T(\sigma,r+1)}},
\qquad
D_0=D+\frac{1}{m_{T(\sigma,D)}}.
\]

Then
\[
A<B<C<D_0.
\]

Let
\[
\varepsilon_n:=\frac{q_{n-1}}{q_n}.
\]

If
\[
\Phi_w(t):=[a_0;a_1\dots,a_n,t],
\]
then
\[
|\Phi_w(x)-\Phi_w(y)|=\frac{|x-y|}{q_n^2(x+\varepsilon_n)(y+\varepsilon_n)}.
\]
Therefore,
\begin{equation}\label{G/T}
    \frac{|G(w,r)|}{|T[wr]|}
=
\frac{C-B}{B-A}\cdot\frac{A+\varepsilon_n}{C+\varepsilon_n},
\end{equation}

\begin{equation}\label{G/J}
    \frac{|G(w,r)|}{|J(w,r+1)|}
=
\frac{C-B}{D_0-C}\cdot\frac{D_0+\varepsilon_n}{B+\varepsilon_n}.
\end{equation}

The RHS of ~(\ref{G/T}) is strictly increasing when $\varepsilon_n>0$. The RHS of ~(\ref{G/J}) is strictly decreasing when $\varepsilon_n>0$.

From the formula
\[
q_k=a_kq_{k-1}+q_{k-2}
\]
if follows that
\[
\varepsilon_k=\frac{1}{a_k+\varepsilon_{k-1}}.
\]
Therefore,
\begin{equation}\label{eps bounds}
    \frac{1}{5}\leq\varepsilon_k\leq1.
\end{equation}

Suppose $\emptyset\neq\sigma=b_1\dots b_s$. Since every admissible prefix starts with $4,3$ and every non-empty suffix starts with $4$, the suffix $\sigma$ starts at least from the second position. Thus, $\varepsilon_{n-s}$ satisfies ~(\ref{eps bounds}). For each non-empty $\sigma$, starting from the interval $[\frac{1}{5},1]$ and applying the function 
\[
x\rightarrow \frac{1}{b_j+x}
\]
to its ends, we can get the bounds for $\varepsilon_n$. The following table lists all possible suffixes $\sigma$, the corresponding
range of $\varepsilon_n$, and the quantity $\lambda_\sigma$, defined as the
maximum of the expressions~(\ref{G/T}) and~(\ref{G/J}) over all admissible
values of $r$. The value $r_{\max}$ denotes the value of $r$ at which this
maximum is attained. The last column gives a strict rational upper bound for
$\lambda_\sigma$.

\[
\renewcommand{\arraystretch}{1.15}
\begin{array}{c|c|c|c|c}
\sigma
&
\text{range of }\varepsilon_n
&
\lambda_\sigma\ \text{(numerical)}
&
r_{\max}
&
\text{exact bound}
\\ \hline

\varnothing
& \frac{1}{5}\leq\varepsilon_n\leq 1
& 0.954098432
& 3
& <\frac{191}{200}
\\

4
& \frac{1}{5}\leq\varepsilon_n\leq\frac{5}{21}
& 0.811003230
& 2
& <\frac{203}{250}
\\

41
& \frac{21}{26}\leq\varepsilon_n\leq\frac{5}{6}
& 0.988509148
& 3
& <\frac{989}{1000}
\\

413
& \frac{6}{23}\leq\varepsilon_n\leq\frac{26}{99}
& 0.951276492
& 3
& <\frac{119}{125}
\\

414
& \frac{6}{29}\leq\varepsilon_n\leq\frac{26}{125}
& 0.810545803
& 2
& <\frac{811}{1000}
\\

4131
& \frac{99}{125}\leq\varepsilon_n\leq\frac{23}{29}
& 0.984037458
& 3
& <\frac{197}{200}
\\

4141
& \frac{125}{151}\leq\varepsilon_n\leq\frac{29}{35}
& 0.817216358
& 2
& <\frac{409}{500}
\\

41314
& \frac{29}{139}\leq\varepsilon_n\leq\frac{125}{599}
& 0.810430951
& 2
& <\frac{811}{1000}
\\

41413
& \frac{35}{134}\leq\varepsilon_n\leq\frac{151}{578}
& 0.951261675
& 3
& <\frac{119}{125}
\\

413141
& \frac{599}{724}\leq\varepsilon_n\leq\frac{139}{168}
& 0.777563310
& 2
& <\frac{389}{500}
\\

414131
& \frac{578}{729}\leq\varepsilon_n\leq\frac{134}{169}
& 0.778043179
& 2
& <\frac{779}{1000}
\end{array}
\]

From this table it is easy to see that the following fact holds.

\begin{lemma}\label{max estimates}
    For all $(\sigma,r)$, except $(\sigma,r)=(41,3)$ the following inequality holds
    \begin{equation}
        \max\left\{\frac{|G(w,r)|}{|T[wr]|},\frac{|G(w,r)|}{|J(w,r+1)|}\right\}<\frac{197}{200}.
    \end{equation}
    For $(\sigma,r)=(41,3)$ the following is true
    \begin{equation}
         \max\left\{\frac{|G(w,r)|}{|T[wr]|},\frac{|G(w,r)|}{|J(w,r+1)|}\right\}<\frac{989}{1000}.
    \end{equation}
\end{lemma}

To prove that $\tau(\ln F_1)>1$ we will need the following lemma.

\begin{lemma}\label{ln lemma}
    Let $E,H$ be two subintervals of a closed interval $P=[u,v]$, where $u>0$. Then
    \[
    \frac{|\ln E|}{|\ln H|}\leq \frac{v}{u}\frac{|E|}{|H|}.
    \]
\end{lemma}
\begin{proof}
    By Lagrange's theorem there exists $\xi_E\in E$ and $\xi_H\in H$, such that
    \[
    |\ln E|=\frac{|E|}{\xi_E}, \ |\ln H|=\frac{|H|}{\xi_H}.
    \]
    Therefore,
    \[
    \frac{|\ln E|}{|\ln H|}=\frac{|E|}{|H|}\frac{\xi_H}{\xi_E}\leq \frac{v}{u}\frac{|E|}{|H|}.
    \]
\end{proof}

\begin{lemma}
    $\tau(\ln F_1)>1$
\end{lemma}
\begin{proof}
    Let $G$ be the deleted interval, and $H$ be any of the two closed intervals left. First assume $(\sigma,r)\neq (41,3)$ Using lemmas ~\ref{max estimates}, ~\ref{ln lemma} and the inequality $\frac{U_F}{L_F}<\frac{253}{250}$, we get
    \[
    \frac{|\ln G|}{|\ln H|}\leq\frac{U_F}{L_F}\frac{|G|}{|H|}<\frac{253}{250}\frac{197}{200}=\frac{48941}{50000}<1.
    \]

    Now we consider the case $\sigma=41, r=3$. The corresponding prefix $w=(a_0,\dots,a_n)$ ends with the digits $4,1$. Thus, 
    \[
    q_n=q_{n-1}+q_{n-2}=(4q_{n-2}+q_{n-3})+q_{n-2}=5q_{n-2}+q_{n-3}\geq 16.
    \]
    For $|T[w]|$ the following holds
    \[
    |T[w]|\leq \frac{1}{q_n(q_n+q_{n-1})}< \frac{1}{q_n^2}\leq \frac{1}{256}.
    \]
    Therefore, $G$ and $H$ are contained in a closed interval $P=[u,v]$, where $u>4$, and $v-u< \frac{1}{256}$. Thus
    \[
    \frac{v}{u}=1+\frac{v-u}{u}<1+\frac{1}{1024}=\frac{1025}{1024}.
    \]
    Now we apply lemma ~\ref{ln lemma} to the intervals $P,G,H$. This yields
    \[
    \frac{|\ln G|}{|\ln H|}\leq \frac{v}{u}\frac{|G|}{|H|}<\frac{1025}{1024}\frac{989}{1000}=\frac{40549}{40960}<1.
    \]
    Therefore,
    \[
    \tau(\ln F_1)>1.
    \]
\end{proof}

\subsection{Proof of theorem \ref{ray upper bound}}

Consider the set
\[
\mathcal{Q}:=\{41414, 4131414, 4141314\}.
\]
Notice that the inverse of every word in $\mathcal{Q}$ also belongs to $\mathcal{Q}$. Let $\mathcal{W}$ be the set of all infinite continued fractions
\[
[a_0;a_1,a_2,\dots], \ a_i\in \{1,2,3,4\},
\]
and the sequence of the partial quotients of does not contain any word from $\mathcal{Q}$. Let
\[
M_4 := \max_{w\in\mathcal{W}} w,
\]
\[
M_3 := \max\left\{
[a_0;a_1,\ldots]\in\mathcal{W} : a_0\leq 3
\right\}.
\]
It is easy to see that
\[
M_4=[\overline{4;1,4,1,3,1,3,1}]=\frac{3(187+\sqrt{67773})}{278},
\]
\[
M_3=[3;\overline{1,4,1,4,1,3,1,3}]=\frac{3(171+\sqrt{67773})}{338}.
\]
Let
\[
C:=M_4M_3=\frac{448875+1611\sqrt{67773}}{46982}.
\]
Then the following inequalities hold
\[
L_F^2<C<10+6\sqrt{2}<U_F^2.
\]
Therefore, for any $\alpha\in [C,10+6\sqrt{2}]$ there exist $X=[x_0;x_1,x_2,\dots], \ Y=[y_0;y_1,y_2,\dots]\in F_1$, such that $\alpha=XY$. Let
\[
S_i:=(x_i,x_{i-1},\dots,x_1,x_0,y_0,y_1,\dots,y_{i-1},y_i),
\]
\[
\xi:=[S_1,2,S_2,2,S_3,2,\dots].
\]
Now we will show that if
\[
\xi=[z_0;z_1,z_2,\dots],
\]
then
\[
d:=d(\xi)=\alpha.
\]
Let $k_i$ be the index of the element $x_0$ inside of $S_i$ that preceeds $y_0$. So $(z_{k_i},z_{k_i+1})=(4,4)$.
Clearly,
\[
d\geq\alpha.
\]

\begin{lemma}\label{partial quotients of xi}
    \begin{itemize}
        \item The sequence of partial quotients of $\xi$ does not contain any words from $\mathcal{Q}$.
        \item All appearances of $4,4$ in $\{z_n\}$ are of the type $(z_{k_i},z_{k_i+1})=(4,4)$.
    \end{itemize}
\end{lemma}
\begin{proof}
    \begin{itemize}
        \item The "left side" of any block $S_i$, which corresponds to $X$ does not contain any words from $\mathcal{Q}$, since the sequence of partial quotients of $X$ does not contain them and $\mathcal{Q}$ is closed under word reversal. Similarly the "right side", which corresponds to $Y$ does not contain them either. A word from $\mathcal{Q}$ cannot be between these two parts, since no word contains $44$. Moreover, it cannot be between blocks $S_i,S_{i+1}$, since no word contains $2$. Therefore, the sequence $\{z_n\}$ does not contain any words from $\mathcal{Q}$.
        \item Inside the left and right side of blocks $S_i$ $44$ does not appear, since $X,Y\in F_1$. It cannot appear between blocks, since the number $2$ is between them. Thus, it can only appear at spots of the type $(z_{k_i},z_{k_i+1})$.
    \end{itemize}
\end{proof}

The following lemma finishes the proof.

\begin{lemma}
    Let $n_i$ be any sequence of natural numbers that approaches infinity, such that there exists
    \[
    d':=\lim_{i\to\infty}d_{n_i}(\xi).
    \]
    Then $d'\leq \alpha$.
\end{lemma}
\begin{proof}
    By lemma ~\ref{partial quotients of xi} $(z_{n_i},z_{n_i+1})\neq(4,4)$, unless $n_i=k_j$. We can assume that this is not the case and $(z_{n_i},z_{n_i+1})=(3,4)$. So, since $\{z_n\}$ does not contain any words from $\mathcal{Q}$, the following inequalities hold.
    \[
    [z_{n_i};z_{n_i-1},\dots,z_1]\leq M_3,
    \]
    \[
    [z_{n_i+1};z_{n_i+2,\dots}]\leq M_4.
    \]
    And the product of these two numbers does not exceed $M_3M_4=C\leq\alpha$.
\end{proof}

\section{Lower bound of the origin of the ray}

\subsection{The plan}

Let
\[
\mathcal{D}(a,b):=\{d(x) \ |\ x\in M(a,b)\}.
\]

It is known that (see \cite{Ivanov1978}) on the interval
\[
\left(
3+\sqrt{3},\,
\frac{9+2\sqrt{21}}{3}
\right)=(4.732\dots,6.055\dots)
\]
$\mathcal{D}$ consists of $\mathcal{D}(3,1)$ and the points
\[
e_n:=4[\overline{(1,3)_n,1,4}]+1.
\]
Using standard methods for comparing two continued fractions it is easy to show that $\{e_n\}$ is strictly increasing. Moreover,
\[
e_0=3+2\sqrt{2}=5.8284\dots,
\]
\[
e_1=\frac{59+4\sqrt{195}}{19}=6.0451\dots
\]
Therefore, this sequence does not intersect with the interval $(L,R)$. All that's left to do now is show that $\mathcal{D}(3,1)$ does not intersect with it either. This is what the rest of this section in dedicated to.

\subsection{Auxiliary facts}

Consider the following set of two-sided sequences
\[
\Sigma:=\{(a_j)_{j\in \mathbb{Z}}\ |\ a_j\in\{1,2,3\},\ 23,32,33\ \text{do not appear in this sequence}\}.
\]
For $\mathbf{a}=(a_j)_{j\in \mathbb{Z}}$ we define
\[
\lambda_j(\mathbf{a}):=[a_j;a_{j-1},a_{j-2},\dots][a_{j+1};a_{j+2},a_{j+3}\dots].
\]
Similarly, we will call a finite or infinite sequence of digits $\{1,2,3\}$ admissible if it does not contain $23,32,33$.

\begin{lemma}\label{two-sided sequence lemma}
    Let $x\in M(3,1)$, and
    \[
    z:=d(x).
    \]
    Then there exists a two-sided sequence $\mathbf{a}=(a_j)_{j\in \mathbb{Z}}\in\Sigma$, such that
    \[
    \lambda_0(\mathbf{a})=z,
    \]
    and for all $j\in \mathbb{Z}$
    \[
    \lambda_j(a)\leq z.
    \]
\end{lemma}
\begin{proof}
    Let $\{n_k\}_{k\in\mathbb{N}}$ be a sequence, such that
    \[
    \lim_{k\to\infty}d_{n_k}(x)=z.
    \]
    Let $\{n_k^1\}_{k\in\mathbb{N}}$ be a subsequence of $\{n_k\}_{k\in\mathbb{N}}$, such that for all $k$ the numbers
    \[
    x_{n_k^1-1},\ x_{n_k^1},\ x_{n_k^1+1}
    \]
    are fixed. Similarly, for any $r\geq2$, let $\{n_k^r\}_{k\in\mathbb{N}}$ be a subsequence of $\{n_k^{r-1}\}_{k\in\mathbb{N}}$, such that for all $k$ the numbers
    \[
    x_{n_k^r-r},\dots,\ x_{n_k^r-1},\ x_{n_k^r},\ x_{n_k^r+1},\dots,\ x_{n_k^r+r}
    \]
    are fixed.

    Let
    \[
    m_k:=n_k^k.
    \]
    Then for all $r$ for all $k$ large enough the numbers
    \[
    x_{m_k-r},\dots,\ x_{m_k-1},\ x_{m_k},\ x_{m_k+1},\dots,\ x_{m_k+r}
    \]
    are fixed. Let 
    \[
    a_j:=x_{m_k+j},
    \]
    where $k$ is sufficiently large, such that this number is fixed. Let
    \[
    \mathbf{a}:=(a_j)_{j\in\mathbb{Z}}\in\Sigma.
    \]

    Let $j\in\mathbb{Z}$ be fixed. Then
    \[
    d_{m_k+j}(x)=[x_{m_k+j};x_{m_k+j-1},\dots x_1][x_{m_k+j+1};x_{m_k+j+2},\dots].
    \]
    By definition of $\{m_k\}$ we have
    \[
    [x_{m_k+j};x_{m_k+j-1},\dots x_1]\to [a_j;a_{j-1},a_{j-2},\dots],
    \]
    \[
    [x_{m_k+j+1};x_{m_k+j+2},\dots]\to [a_{j+1};a_{j+2},\dots],
    \]
    as $k\to\infty$. Therefore,
    \[
    \lambda_j(\mathbf{a})=\lim_{k\to\infty}d_{m_k+j}(x)\leq z.
    \]
    Moreover,
    \[
    \lambda_0(\mathbf{a})=\lim_{k\to\infty}d_{m_k}(x)=z.
    \]
\end{proof}

For $j\in\{1,2,3\}$ let $X_j$ be the set of all numbers of the form
\[
[j;b_1,b_2,\dots],
\]
such that the sequence $j,b_1,b_2,\dots$ is admissible. Let
\[
m_j:=\min X_j, \ M_j:=\max X_j.
\]
We have
\[
m_1=[1;\overline{3,1}]=\frac{3+\sqrt{21}}{6}, \ M_1=[1;\overline{1,3}]=\frac{-1+\sqrt{21}}{3},
\]
\[
m_2=[2;2,\overline{1,3}]=\frac{19+\sqrt{21}}{10}, \ M_2=[2;\overline{1,3}]=\frac{1+\sqrt{21}}{2},
\]
\[
m_3=[3;1,\overline{1,3}]=\frac{31+\sqrt{21}}{10}, \ M_3=[3;\overline{1,3}]=\frac{3+\sqrt{21}}{2}.
\]

We can continue any finite admissible sequence $u=(u_0,u_1,\dots,u_m)$ to an infinite admissible sequence. If $u_m=d$ we denote by $J_d$ the closed interval $[\alpha,\beta]$, such that $\alpha=[a_0;a_1,\dots]$ is the minimal number for which the sequence $u_0,u_1,\dots,u_m,a_0,a_1,\dots$ is admissible. And $\beta=[b_0;b_1,\dots]$ is the maximal number, such that the sequence $u_0,u_1,\dots,u_m,b_0,b_1,\dots$ is admissible. We have
\[
J_1=[m_1,M_3],
\]
\[
J_2=[m_1,M_2],
\]
\[
J_3=[m_1,M_1].
\]

For an admissilbe sequence $u=(u_0,u_1,\dots,u_m)$ let
\[
J_{u_m}=[\alpha,\beta].
\]
We can easily find the interval of all values of the expression
\[
[u_0;u_1,\dots,u_m,t], \ t\in J_{u_m}
\]
using the following algorithm. We start off with the closed interval
\[
[l,h]:=[\alpha,\beta].
\]
Then for $i=m,m-1,\dots,0$ we replace
\[
[l,h]\xrightarrow{}\left[u_i+\frac{1}{h},u_i+\frac{1}{l}\right].
\]
Since all the functions of the type $t\xrightarrow{}u_i+\frac{1}{t}$ are strictly decreasing for $t>0$, the final interval $[l,h]$ will be our target. We will denote this interval $I(u)$.

Consider an admissible sequence
\[
w=(w_{-r+1},\dots,w_0,w_1,\dots,w_r)
\]
of length $2r$. For $p\in\{-r+1,\dots,r\}$ we define
\[
I_p^-:=I(w_p,w_{p-1},\dots,w_{-r+1}),
\]
\[
I_p^+(w)=
\begin{cases}
I(w_{p+1},w_{p+2},\ldots,w_r), & p<r,\\[1mm]
J_{w_r}, & p=r.
\end{cases}
\]
Let
\[
I_p^-(w)=[l_p^-,h_p^-], \ I_p^+(w)=[l_p^+,h_p^+].
\]
And 
\[
\underline{\lambda}_p(w):=l_p^-l_p^+, \ \overline{\lambda}_p(w):=h_p^-h_p^+.
\]
Then the following lemma holds
\begin{lemma}\label{continuation ineq lemma}
    Let $\mathbf{a}$ be any two-sided admissible continuation of $w=(w_{-r+1},\dots,w_0,w_1,\dots,w_r)$. Then for any $p\in\{-r+1,\dots,r\}$ the following inequalities hold
    \[
    \underline{\lambda}_p(w)\leq\lambda_p(\mathbf{a})\leq\overline{\lambda}_p(w).
    \]
\end{lemma}
\begin{proof}
    Let
    \[
    \mathbf{a}=(\dots,b_1,b_0,w_{-r+1},\dots,w_{p-1},w_p,w_{p+1},\dots,w_r,a_0,a_1,\dots).
    \]
    Then
    \[
    \lambda_p(\mathbf{a})=[w_p;w_{p-1},\dots,w_{-r+1},b_0,b_1,\dots][w_{p+1};\dots,w_r,a_0,a_1,\dots].
    \]
    By definition we have
    \[
    l_p^-\leq[w_p;w_{p-1},\dots,w_{-r+1},b_0,b_1,\dots]\leq h_p^-,
    \]
    \[
    l_p^+\leq [w_{p+1};\dots,w_r,a_0,a_1,\dots]\leq h_p^+.
    \]
    Multiplying these inequalities yields the wanted result.
\end{proof}

\begin{definition}
    We will call an admissible sequence
    \[
    w=(w_{-r+1},\dots,w_0,w_1,\dots,w_r)
    \]
    surviving if
    \[
    \overline{\lambda}_0(w)>L, \text{and} \ \underline{\lambda}_p(w)<R \ (-r+1\leq p\leq r).
    \]
    By $S_r$ we will denote the set of all surviving sequences of length $2r$.
\end{definition}

\begin{lemma}\label{inheritence lemma}
    Let $v=bwc$ be an admissible sequence, where $b,c$ are digits. If $v$ survives, then $w$ survives.
\end{lemma}
\begin{proof}
    For all $p$ we have
    \[
    \underline{\lambda}_p(w)\leq\underline{\lambda}_p(v), \ \overline{\lambda}_p(w)\geq\overline{\lambda}_p(v).
    \]
    Therefore,
    \[
    L<\overline{\lambda}_0(v)\leq\overline{\lambda}_0(w),
    \]
    and
    \[
    \underline{\lambda}_p(w)\leq\underline{\lambda}_p(v)<R.
    \]
    Thus, $w$ survives.
\end{proof}

\begin{lemma}\label{surviving word lemma}
    Let $\mathbf{a}\in\Sigma$ satisfy
    \[
    L<\lambda_0(\mathbf{a})<R, \ \lambda_j(\mathbf{a})<R \ (j\in \mathbb{Z}).
    \]
    Then for all $r\geq 1$ the sequence
    \[
    w=(a_{-r+1},\dots,a_0,a_1,\dots,a_r)
    \]
    survives.
\end{lemma}
\begin{proof}
    By lemma ~\ref{continuation ineq lemma} we have
    \[
    L<\lambda_0(\mathbf{a})\leq\overline{\lambda}_0(w),
    \]
    and
    \[
    \underline{\lambda}_p(w)\leq\lambda_p(\mathbf{a})<R
    \]
    for all $p\in \{-r+1,\dots,r\}$. Therefore, $w$ survives.
\end{proof}

Using these lemmas it is easy to find $S_r$ recursively. We start with $r=1$. The possible candidates for surviving sequences are
\[
11,12,13,21,22,31.
\]
If we have candidates for $S_r$ it is easy to see if one of them, say $w$ survives. For that we just need to check the inequalities
\[
\overline{\lambda}_0(w)>L, \ \underline{\lambda}_p(w)<R \ (-r+1\leq p\leq r).
\]
Both $\overline{\lambda}_0(w)$ and $\underline{\lambda}_p(w)$ can be easily calculated using the algorithm provided earlier. Then if we have the full set $S_r$, to find the set $S_{r+1}$ by lemma ~\ref{inheritence lemma} we just have to consider the admissible sequences of the form $bwc$, where $w\in S_r$.

For a finite sequence $u=(u_0,u_1,\dots,u_m)$ we denote by $u^*$ its inverse, so the sequence $(u_m,\dots,u_1,u_0)$. Using the provided algorithm we can run a Python program (see Appendix~\ref{app:code}) that shows that the only sequences from $S_{17}$ are
\[
2112111311113111131112112131313111,
\]
\[
2112111311113111131112112131313121,
\]
\[
2112111311113111131112112131313131,
\]
or their inverses. It is easy to see that the block
\[
C:=21121113111131111311121121=c_0c_1\dots c_{25},
\]
or $C^*$ is contained in every sequence from $S_{17}$.

\begin{lemma}\label{max lemma}
    Suppose $\mathbf{a}=(a_j)_{j\in\mathbb{Z}}\in\Sigma$ and for some $s$
    \[
    a_{s+i}=c_i, \text{or} \ a_{s+i}=c_i^*.
    \]
    Then for some $j$
    \[
    \lambda_j(\mathbf{a})\geq R.
    \]
\end{lemma}
\begin{proof}
    First assume that $a_{s+i}=c_i$. Let
    \[
    x:=[a_{s-1};a_{s-2},\dots], \ y:=[a_{s+26};a_{s+27},\dots].
    \]
    Since,
    \[
    \lambda_{s+7}(\mathbf{a})=[a_{s+7};a_{s+6},\dots,a_s,x][a_{s+8};a_{s+9},\dots,a_{s+25},y],
    \]
    and
    \[
    \lambda_{s+16}(\mathbf{a})=[a_{s+16};a_{s+15},\dots,a_s,x][a_{s+17};a_{s+18},\dots,a_{s+25},y].
    \]
    It can be easily calculated that
    \[
    A_7(x):=[a_{s+7};a_{s+6},\dots,a_s,x]=\frac{178x+69}{49x+19},
    \]
    \[
    B_7(y):=[a_{s+8};a_{s+9},\dots,a_{s+25},y]=\frac{30287y+21826}{18460y+13303},
    \]
    \[
    A_{16}(x):=[a_{s+16};a_{s+15},\dots,a_s,x]=\frac{21826x+8461}{13303x+5157},
    \]
    \[
    B_{16}(y):=[a_{s+17};a_{s+18},\dots,a_{s+25},y]=\frac{247y+178}{68y+49}.
    \]
    Let
    \[
    P(x):=A_7(x)A_{16}(x), \ Q(y):=B_7(y)B_{16}(y).
    \]
    Differentiation shows that
    \[
    P'(x)=
    \frac{6\bigl(48390426x^2+37517767x+7272011\bigr)}
    {(49x+19)^2(13303x+5157)^2}
    >0
    \qquad (x>0).
    \]
    And
    \[
    Q'(y)=
    -\frac{6\bigl(93180204y^2+134298619y+48390426\bigr)}
    {(68y+49)^2(18460y+13303)^2}
    <0
    \qquad (y>0).
    \]
    Clearly $x\geq m_1$, $y\leq M_3$. Therefore
    \[
    P(x)\geq P(m_1), \ Q(y)\geq Q(M_3).
    \]
    So,
    \[
    \lambda_{s+7}(\mathbf{a})\lambda_{s+16}(\mathbf{a})=P(x)Q(y)\geq P(m_1)Q(M_3)=R^2.
    \]
    Therefore,
    \[
    \max \{\lambda_{s+7}(\mathbf{a}),\lambda_{s+16}(\mathbf{a})\}\geq R.
    \]

    Now we prove the case $a_{s+i}=c_i^*$. Consider the two-sided sequence $\mathbf{b}=(b_j)_{j\in\mathbb{Z}}\in\Sigma$, such that $b_j=a_{-j}$. We have
    \[
    \lambda_j(\mathbf{b})=[b_j;b_{j-1},\dots][b_{j+1};b_{j+2},\dots]=[a_{-j};a_{-j+1},\dots][a_{-j-1};a_{-j-2},\dots]=\lambda_{-j-1}(\mathbf{a}).
    \]
    For $\mathbf{b}$ we have
    \[
    b_{-s-25}=a_{s+25}=c_{25}^*=c_0,b_{-s-24}=a_{s+24}=c_{24}^*=c_1,\dots,b_{-s}=a_s=c_0^*=c_{25}.
    \]
    Let $s':=-s-25$. Then by the first part of this lemma we have
    \[
    R\leq\max \{\lambda_{s'+7}(\mathbf{b}),\lambda_{s'+16}(\mathbf{b})\}=\max \{\lambda_{-s-18}(\mathbf{b}),\lambda_{-s-9}(\mathbf{b})\}=\max \{\lambda_{s+17}(\mathbf{a}),\lambda_{s+8}(\mathbf{a})\}.
    \]
\end{proof}

\begin{lemma}\label{main lemma}
    There does not exist $\mathbf{a}=(a_j)_{j\in\mathbb{Z}}\in\Sigma$, such that
    \[
    L<\lambda_0(\mathbf{a}), \ \lambda_j(\mathbf{a})<R \ (j\in\mathbb{Z}).
    \]
\end{lemma}
\begin{proof}
    Assume that such a two-sided sequence $\mathbf{a}$ exists. Then by lemma ~\ref{surviving word lemma} its central subsequence of length $34$ (meaning the sequence $a_{-16},\dots,a_0,a_1,\dots,a_{17}$), survives. Therefore $\mathbf{a}$ contains $C$ as a subsequence. Then by lemma ~\ref{max lemma} for some $j$ we have $\lambda_j(\mathbf{a})\geq R$. Which is a contradiction.
\end{proof}

\subsection{Proof of theorem \ref{ray lower bound}}

Now the proof of theorem ~\ref{ray lower bound} is trivial. We have to prove that
\[
\mathcal{D}(3,1)\cap (L,R)=\emptyset.
\]
Assume that this is false. Then there exists $x\in M(3,1)$, such that
\[
z:=d(x)\in (L,R).
\]
By lemma ~\ref{two-sided sequence lemma} we can construct a two-sided sequence $\mathbf{a}\in\Sigma$, such that
\[
\lambda_0(\mathbf{a})=z,
\]
and
\[
\lambda_j(\mathbf{a})\leq z<R \ (j\in\mathbb{Z}).
\]
However, by lemma \ref{main lemma} this is impossible. And so we arrive at a contradiction.

\section{Isolated points in Dirichlet spectrum}

\subsection{Preliminaries}

To prove this we will use the results from \cite{Ivanov1978}. We define the numbers
\[
\nu(a,b)=
\begin{cases}
a[\overline{1;a-1}]+1, & \text{if } b=1,\\[2mm]
a[b;a-1,\overline{1,a}]+1, & \text{if } b\ge 2,
\end{cases}
\]
\[
\bar{\nu}(a,b)=
\begin{cases}
[a-1;\overline{1,a}]^2, & \text{if } b\le a-2,\\[2mm]
[\overline{a;1}]\,[b;\overline{1,a}], & \text{if } b\ge a-1.
\end{cases}
\]
These two points are the smallest and largest limit points of $\mathcal{D}(a,b)$, and $\max \mathcal{D}(a,b)=\bar{\nu}(a,b)$. Moreover, the following holds
\[
\mathcal{D}(a,b)\cap(0,\nu(a,b))
=
\begin{cases}
\left\{d\!\left(x^{(n,a)}\right)\right\}_{n=0}^{\infty},
& \text{if } b=1,\\[2mm]
\left\{d\!\left([\overline{a;b}]\right)\right\},
& \text{if } b\ge 2,
\end{cases}
\]
where
\[
x^{(n,a)}
=
[\overline{a;(1,a-1)_n,1}].
\]
For further calculations we will also need
\begin{equation}\label{D(xna)}
    D(x^{(n,a)})=a[\overline{(1;a-1)_n,1,a}]+1,
\end{equation}
\begin{equation}\label{min D(a,b)}
    \min\mathcal{D}(a,b)=b[\overline{a,b}]+1
\end{equation}
for $b\geq2$, and
\begin{equation}\label{min D(a)}
    \min \mathcal{D}(a)=[\overline{a,1}]+1.
\end{equation}

\subsection{Proof of theorem \ref{isolated points}}

As mentioned before, it suffices to show that $P_k< Z_k$ for all $k\geq 4$. Let
\[
n:=2k-3, \quad u:=F_n, \quad v:=F_{n+1}.
\]
Using Cassini's identity and the fact that $n$ is odd, we get
\[
F_{n-1}F_{n+1}-F_n^2=(-1)^n=-1 \implies (v-u)v-u^2=-1,
\]
or
\begin{equation}\label{Cassini}
    v^2=uv+u^2-1.
\end{equation}

For further use we also introduce
\begin{equation}\label{tau}
\tau:=[\overline{1,3}]=\frac{3+\sqrt{21}}{6} \implies 3\tau^2-3\tau-1=0.
\end{equation}

For a positive integer $a$, set
\[
T_a:=
\begin{pmatrix}
a & 1\\
1 & 0
\end{pmatrix}.
\]
If
\[
T_{a_0}\cdots T_{a_r}
=
\begin{pmatrix}
A & B\\
C & D
\end{pmatrix},
\]
then
\[
[a_0;a_1,\ldots,a_r,t]
=
\frac{At+B}{Ct+D}.
\]
Moreover,
\begin{equation}\label{Fibonacci matrix}
T_1^m
=
\begin{pmatrix}
F_{m+1} & F_m\\
F_m & F_{m-1}
\end{pmatrix}.
\end{equation}

\subsubsection{Calculating $P_k$}

Using ~\ref{Fibonacci matrix}, we obtain
\[
T_3T_1^{\,n-2}T_2
=
\begin{pmatrix}
u+3v & 2v-u\\
v & v-u
\end{pmatrix}.
\]
Therefore,
\[
[3,(1)_{n-2},2,\tau]=\frac{N_1}{D_1},
\]
where
\[
N_1=(u+3v)\tau+2v-u,
\]
\[
D_1=v\tau+v-u.
\]
Using \ref{Cassini}, one can show that
\[
T_1^n T_3 T_1^n T_2 T_1^2 T_2
=
\begin{pmatrix}
49u^2+80uv-44 & 19u^2+31uv-17\\
31u^2+49uv-13 & 12u^2+19uv-5
\end{pmatrix}.
\]
Hence,
\[
[(1)_n,3,(1)_n,2,1,1,2,\tau]=\frac{N_2}{D_2},
\]
where
\[
N_2=(49u^2+80uv-44)\tau+19u^2+31uv-17,
\]
\[
D_2=(31u^2+49uv-13)\tau+12u^2+19uv-5.
\]
Using the formula \ref{Pk}, we obtain
\begin{equation}\label{Pk closed form}
    P_k=\frac{N_1N_2}{D_1D_2}.
\end{equation}

\subsubsection{Calculating $Z_k$}

Let
\[
y:=[\overline{(1)_{n+2},3}].
\]
Then
\[
Z_k=1+3y.
\]
From
\[
T_1^{n+2}T_3
=
\begin{pmatrix}
4u+7v & u+2v\\
3u+4v & u+v
\end{pmatrix}.
\]
we get that $y$ satisfies the following quadratic equation
\[
(3u+4v)y^2-3(u+2v)y-(u+2v)=0.
\]
And so $Z_k$ is the largest root of the polynomial
\[
H_k(X):=(3u+4v)(X-1)^2-9(u+2v)X=(3u+4v)X^2-(15u+26v)X+(3u+4v).
\]
Thus
\begin{equation}\label{Zk closed form}
    Z_k=\frac{15u+26v+3\sqrt{(u+2v)(21u+34v)}}{2(3u+4v)}.
\end{equation}
For $H_k$ we have
\[
H_k(0)=3u+4v>0, \quad H_k(1)=-9(u+2v)<0.
\]
Therefore, the smaller root of $H_k$ lies in the interval $(0,1)$. The product of the two root is $1$ by Viet's theorem. Hence, $Z_k>1$, and for any $X>1$ we have
\begin{equation}\label{Hk<0}
    H_k(X)<0 \iff X<Z_k.
\end{equation}
This fact will be crucial for our proof.

\subsubsection{Proof of the inequality $P_k<Z_k$}

Now we finish the proof of theorem \ref{isolated points} by showing that $P_k<Z_k$ for $k\geq 4$.

Let $D:=D_1D_2$. Substituting (\ref{Pk closed form}) into $H_k$ and using the formulas (\ref{Cassini}), (\ref{tau}) we obtain
\begin{equation}\label{Hk(Pk)}
9D^2H_k(P_k)=-(\tau E_1(u,v)+E_0(u,v)),
\end{equation}
where
\[
E_1(u,v)
=
u^3\bigl(406224u^2-2556255\bigr)
+
u^2v\bigl(658188u^2-3630069\bigr)
+
2128467u
+
1097730v,
\]
and
\[
E_0(u,v)
=
u^3\bigl(102732u^2-666542\bigr)
+
u^2v\bigl(166464u^2-950107\bigr)
+
558227u
+
289354v.
\]

Since $k\geq 4$, we have $n=2k-3\geq 5$. Therefore, $u=F_n\geq F_5=5$. Thus,
\[
406224u^2-2556255 \ge 7599345 > 0,
\]
\[
658188u^2-3630069 \ge 12824631 > 0,
\]
\[
102732u^2-666542 \ge 1901758 > 0,
\]
\[
166464u^2-950107 \ge 3211493 > 0.
\]
Since $u,v,\tau>0$ it follows that
\[
\tau E_1(u,v)>0, \quad E_0(u,v)>0.
\]
Consequently, from (\ref{Hk(Pk)}) we obtain $H_k(P_k)<0$. Since $P_k>1$ this implies $P_k<Z_k$ from (\ref{Hk<0}).

\begin{corollary}
    The isolated points $Z_k$ form a strictly increasing sequence, such that
    \[
    Z_k\to Z_\infty:=\frac{5+3\sqrt{5}}{2}.
    \]
    Consequently, $Z_\infty$ is an accumulation point of isolated points of the Dirichlet spectrum.
\end{corollary}
\begin{proof}
    From formula (\ref{Zk}) it is easy to see that this sequence is strictly increasing. Moreover,
    \[
    Z_k\to 3[\overline{1}]+1=3\frac{1+\sqrt{5}}{2}+1=\frac{5+3\sqrt{5}}{2}=Z_\infty.
    \]
\end{proof}

\subsection{Proof of theorem \ref{asymptotics}}

For the sake of convenience we denote $\Lambda:=Z_\infty$. Once again,
\[
n:=2k-3, \quad u:=F_n,\quad v:=F_{n+1}.
\]

First we find the asymptotic formulas for $P_k,Z_k,Q_k$.

Let
\[
s:=\frac{1}{u^2},\quad r:=\frac{v}{u}.
\]
Dividing (\ref{Cassini}) by $u^2$ yields
\[
r^2-r-1=-s.
\]
Since $r>0$ we get
\[
r=\frac{1+\sqrt{5-4s}}{2}.
\]
Let $\varphi$ be the golden ratio. Then the Taylor expansion at point $s=0$ gives us
\begin{equation}\label{r}
    r=\varphi-\frac{s}{\sqrt{5}}+O(s^2).
\end{equation}

We first treat $Z_k$. Dividing the numerator and denominator in (\ref{Zk closed form}) by $u$ yields
\begin{equation}\label{z(r)}
    Z_k=z(r):=\frac{15+26r+3\sqrt{(1+2r)(21+34r)}}{2(3+4r)}.
\end{equation}
It can be calculated that 
\[
z(\varphi)=\Lambda, \quad z'(\varphi)=39\sqrt{5}-87.
\]
The Taylor expansion of $z(r)$ and (\ref{r}) yield
\begin{equation}\label{Zk Taylor}
    Z_k=z(\varphi)+z'(\varphi)(r-\varphi)+O((r-\varphi)^2)=\Lambda-\frac{39\sqrt{5}-87}{\sqrt{5}}s+O(s^2)=\Lambda-A_Zs+O(s^2).
\end{equation}

Now we consider $Q_k$. From
\[
T_1^{n+2}T_2
=
\begin{pmatrix}
3u+5v & u+2v\\
2u+3v & u+v
\end{pmatrix}.
\]
we get
\[
Q_k=1+3\frac{(3u+5v)\tau+u+2v}{(2u+3v)\tau+u+v}.
\]
Dividing by $u$ yields
\begin{equation}\label{q(r)}
    Q_k=q(r):=1+3\frac{(3+5r)\tau+1+2r}{(2+3r)\tau+1+r}.
\end{equation}
Calculations give
\[
q(\varphi)=\Lambda,\quad \frac{q'(\varphi)}{\sqrt{5}}=\frac{-117-21\sqrt{21}+9\sqrt{105}}{16}+\frac{273\sqrt{5}}{80}=:A_Q.
\]
Using Taylor expansion of (\ref{q(r)}) and (\ref{r}) we obtain
\begin{equation}\label{Qk Taylor}
    Q_k=q(\varphi)+q'(\varphi)(r-\varphi)+O((r-\varphi)^2)=\Lambda-A_Qs+O(s^2).
\end{equation}

Finally we deal with $P_k$. Dividing (\ref{Pk closed form}) by appropriate powers of $u$ we get $P_k$ as a function of $r,s$
\begin{equation}\label{p(r,s)}
P_k=p(r,s):=\frac{(1+3r)\tau+2r-1}{r\tau+r-1}\frac{(49+80r-44s)\tau+19+31r-17s}{(31+49r-13s)\tau+12+19r-5s}.
\end{equation}
Calculations show
\[
p(\varphi,0)=\Lambda,\quad \frac{p_r(\varphi,0)}{\sqrt{5}}-p_s(\varphi,0)=\frac{11+\sqrt{105}}{8}+\frac{\sqrt{21}}{4}-\frac{33\sqrt{5}}{20}=:A_P.
\]
The Taylor expansion of (\ref{p(r,s)}) and (\ref{r}) yields
\begin{equation}\label{Pk Taylor}
    P_k=p(r,s)=\Lambda+p_r(\varphi,0)(r-\varphi)+p_s(\varphi,0)s+O((r-\varphi)^2+s^2)=\Lambda-A_Ps+O(s^2).
\end{equation}

Now, considering $s=F_{2k-3}^{-2}$, we obtain
\[
Z_k-P_k=(A_P-A_Z)F_{2k-3}^{-2}+O(F_{2k-3}^{-4}),
\]
\[
Q_k-Z_k=(A_Z-A_Q)F_{2k-3}^{-2}+O(F_{2k-3}^{-4}),
\]
\[
Q_k-P_k=(A_P-A_Q)F_{2k-3}^{-2}+O(F_{2k-3}^{-4}).
\]
Simple calculations give
\[
A_P-A_Z=C_-,\quad A_Z-A_Q=C_+,\quad A_P-A_Q=C.
\]
And this yields (\ref{Zk-Pk}), (\ref{Qk-Zk}), (\ref{Qk-Pk}), which finishes proof of theorem \ref{asymptotics}.

Using this one can easily prove the following fact.
\begin{corollary}
    \[
    Z_k-P_k\sim 5C_-\varphi^{-4k+6},
    \]
    \[
    Q_k-Z_k\sim 5C_+\varphi^{-4k+6},
    \]
    \[
    Q_k-P_k\sim5C\varphi^{-4k+6}.
    \]
    Particularly,
    \[
    \frac{Z_k-P_k}{Q_k-Z_k}\to \frac{C_-}{C_+}.
    \]
\end{corollary}
\begin{proof}
    Using Binet's formula
    \[
    F_n=\frac{\varphi^n-\psi^n}{\sqrt{5}},
    \]
    where $\psi=\frac{1-\sqrt{5}}{2}$, we get
    \[
    Z_k-P_k=\frac{C_-}{\frac{\varphi^{2n}-2\varphi^n\psi^n+\psi^{2n}}{5}}+O\left(\frac{25}{(\varphi^n-\psi^n)^4}\right)\sim 5C_-\varphi^{-4k+6}.
    \]
    The other equations are proven similarly.
\end{proof}

\subsection{Proof of theorem \ref{equivalence classes}}

In this section we prove theorem \ref{equivalence classes}. We split the proof into 3 lemmas.

\begin{lemma}\label{global localization}
    Let $k\geq2$ and for some real $x$ $d(x)=Z_k$. If $k\neq3$, then
    \[
    x\in M(3,1).
    \]
    If $k=3$, then the only additional possibilities are
    \[
    x\in M(2,2),  \text{or} \ x\in M(4,1). 
    \]
    In these cases
    \[
    x\sim [\overline{2}], \text{or} \ x\sim [\overline{4,1}] 
    \]
    respectively.
\end{lemma}
\begin{proof}
    Let $x\in M(a)$ for some $a\geq 1$. We consider the possible values of $a$.

    Let $a=1$. Then
    \[
    d(x)=\frac{3+\sqrt{5}}{2}<Z_2.
    \]
    Therefore, $a\neq 1$

    Let $a=2$. Using beforementioned results we get
    \[
    \max \mathcal{D}(2,1)=\bar{v}(2,1)=[\overline{2;1}][1;\overline{1,2}]=3+\sqrt{3}<Z_2,
    \]
    and
    \[
    v(2,2)=2[2;1,\overline{1,2}]+1=5+\frac{2}{\sqrt{3}}>6.
    \]
    Hence the only possibility in this case is
    \[
    x\sim [\overline{2}].
    \]
    For this $x$ we have
    \[
    d(x)=3+2\sqrt{2}=Z_3.
    \]

    Let $a=3$. For $b\geq2$ we have
    \[
    \min \mathcal{D}(3,b)=b[\overline{3,b}]+1\geq 3b+1\geq 7.
    \]
    Thus, if $x\in M(3)$ then $x\in M(3,1)$.

    Let $a=4$. Then if $b\geq2$, we once again have
    \[
    \min \mathcal{D}(4,b)\geq 4b+1\geq 9.
    \]
    For $b=1$ we get
    \[
    v(4,1)=4[\overline{1,3}]+1=\frac{9+2\sqrt{21}}{3}>6.
    \]
    For the increasing discrete sequence below $v(4,1)$ we have, using (\ref{D(xna)})
    \[
    D(x^{(0,4)})=D([\overline{4,1}])=3+2\sqrt{2}=Z_3,
    \]
    \[
    D(x^{(1,4)})=\frac{59+4\sqrt{195}}{19}>6.
    \]
    Therefore, in this case the only point coming from $M(4,1)$ is $Z_3$, and in this case $x\sim [\overline{4,1}]$, since as stated in \cite{Ivanov1978}, the minimal points have only one equivalence class.

     Finally we consider the case $a\geq 5$. Using (\ref{min D(a)}) we get
    \[
    \min \mathcal{D}(a)=[\overline{a,1}]+1\geq a+1\geq6.
    \]
    And so $a$ cannot be bigger than $4$. This concludes the proof of the lemma.
\end{proof}

\begin{lemma}\label{localization inside}
    Let
    \[
    x\in M(3,1), \quad d(x)=Z_k.
    \]
    Then
    \[
    x\in M_{k,0}.
    \]
\end{lemma}
\begin{proof}
    Ivanov's proof in \cite{Ivanov1982} shows that if $x\in M(3,1)$ and $d(x)<\Lambda$, then $x$ belongs to one of the following sets
    \[
    M_{1,1},\quad M_{j,0}, M_{j,1} \ (j\geq 2).
    \]
    In this article we proved that
    \[
    \mathfrak{m}_{1,1}\prec\mathfrak{m}_{2,0}\prec\mathfrak{m}_{2,1}\prec\mathfrak{m}_{3,0}\prec\dots
    \]
    The proof trivially follows from this fact.
\end{proof}

\begin{lemma}\label{rigidity}
    Let
    \[
    x\in M_{k,0}, \quad d(x)=Z_k.
    \]
    Then
    \[
    x\sim [\overline{(1)_{2k-1},3}].
    \]
\end{lemma}
\begin{proof}
    From Ivanov's proposition $1$ in \cite{Ivanov1982} we get that if $x\in M_{k,0}$ and $d(x)<Q_k$, then $x\sim[\overline{(1)_{2k-1},3}]$. The proof follows immediately, since $d(x)=Z_k<Q_k$.
\end{proof}

Now we prove theorem \ref{equivalence classes}.
\begin{proof}[Proof of Theorem \ref{equivalence classes}]
    Let $d(x)=Z_k$ for some real $x$. If $k\neq3$, then $x\in M(3,1)$ by lemma \ref{global localization}. Now by lemma \ref{localization inside} we get $x\in M_{k,0}$. Finally by lemma \ref{rigidity} we have $x\sim[\overline{(1)_{2k-1},3}]$.

    Now consider the case $k=3$. By lemma \ref{global localization} we have $x\in M(3,1)$, or $x\in M(2,2)$, or $x\in M(4,1)$. The first case gives us
    \[
    x\sim[\overline{(1)_5,3}]
    \]
    by lemmas \ref{localization inside} and \ref{rigidity}. And the other two cases give us
    \[
    x\sim [\overline{2}], \text{or} \ x\sim [\overline{4,1}]
    \]
    respectively by lemma \ref{global localization}. This concludes the proof.
\end{proof}

\appendix
\section{Source code for the finite search}
\label{app:code}

The following Python program was used for the exact finite search
appearing in the proof of Theorem~\ref{ray lower bound}.
All arithmetic in $\mathbb Q(\sqrt{21})$ is exact.

\begin{lstlisting}[
    language=Python,
    basicstyle=\ttfamily\footnotesize,
    breaklines=true,
    numbers=left,
    numberstyle=\tiny
]
from fractions import Fraction

class Q21:
    def __init__(self, a, b):
        self.a = a
        self.b = b

    def __add__(self, other):
        return Q21(self.a + other.a, self.b + other.b)

    def __sub__(self, other):
        return Q21(self.a - other.a, self.b - other.b)

    def __neg__(self):
        return Q21(-self.a, -self.b)

    def __mul__(self, other):
        return Q21(self.a * other.a + 21 * self.b * other.b, self.b * other.a + self.a * other.b)

    def __truediv__(self, other):
        d = other.a ** 2 - 21 * other.b ** 2
        x = self * Q21(other.a, -other.b)
        return Q21(x.a/d, x.b/d)

    def __le__(self, other):
        x = other - self
        if x.a >= 0 and x.b >= 0:
            return True
        if x.a <= 0 and x.b <= 0:
            return False
        if x.a < 0 and x.b > 0:
            return 21 * x.b ** 2 >= x.a ** 2
        if x.a > 0 and x.b < 0:
            return x.a ** 2 >= 21 * x.b ** 2


m1 = Q21(Fraction(3, 6), Fraction(1, 6))
M1 = Q21(Fraction(-1, 2), Fraction(1, 2))
m2 = Q21(Fraction(19, 10), Fraction(1, 10))
M2 = Q21(Fraction(1, 2), Fraction(1, 2))
m3 = Q21(Fraction(31, 10), Fraction(1, 10))
M3 = Q21(Fraction(3, 2), Fraction(1, 2))
L = Q21(Fraction(7), Fraction(-8, 35))
R = Q21(Fraction(381780528649, 64088662090), Fraction(35642709, 64088662090))

def min_interval(d):
    if d == 1:
        return [m1, M3]
    elif d == 2:
        return [m1, M2]
    elif d == 3:
        return [m1, M1]

def get_interval_of_word(u):
    m = len(u)
    interval = min_interval(u[m-1])
    for i in range(m):
        interval[0], interval[1] = (
            Q21(1, 0) / interval[1] + Q21(u[m - 1 - i], 0),
            Q21(1, 0) / interval[0] + Q21(u[m - 1 - i], 0)
        )
    return interval

def get_lower_lambda(w, p):
    lower_interval = get_interval_of_word(w[0:(p + 1)][::-1])
    if p < len(w) - 1:
        upper_interval = get_interval_of_word(w[p + 1:])
    else:
        upper_interval = min_interval(w[-1])
    lower_lambda = lower_interval[0] * upper_interval[0]
    return lower_lambda

def get_upper_lambda(w, p):
    lower_interval = get_interval_of_word(w[0:(p + 1)][::-1])
    if p < len(w) - 1:
        upper_interval = get_interval_of_word(w[p + 1:])
    else:
        upper_interval = min_interval(w[-1])
    upper_lambda = lower_interval[1] * upper_interval[1]
    return upper_lambda

def does_word_survive(w):
    if get_upper_lambda(w, int(len(w) / 2 - 1)) <= L:
        return False
    else:
        for p in range(len(w)):
            if get_lower_lambda(w, p) >= R:
                return False
    return True

def get_allowed_digits(d):
    if d == 1:
        return [1, 2, 3]
    elif d == 2:
        return [1, 2]
    elif d == 3:
        return [1]

def get_survivable_words(r):
    survivable_words = []
    candidates = [[1, 1], [1, 2], [1, 3], [2, 1], [2, 2], [3, 1]]
    candidates_temp = []
    print('2r', 'num of candidates', '|S_r|', 'survived words')
    for k in range(1, r + 1):
        for candidate in candidates:
            if does_word_survive(candidate):
                survivable_words.append(candidate)
        print(2 * k, len(candidates), len(survivable_words), survivable_words)
        candidates_temp = survivable_words.copy()
        survivable_words = []
        candidates = []
        for temp_candidate in candidates_temp:
            left_digits = get_allowed_digits(temp_candidate[0])
            right_digits = get_allowed_digits(temp_candidate[-1])
            for left_digit in left_digits:
                for right_digit in right_digits:
                    temp = temp_candidate.copy()
                    temp.insert(0, left_digit)
                    temp.append(right_digit)
                    candidates.append(temp)

get_survivable_words(17)
\end{lstlisting}

\section*{AI assistance}

The proofs of the results presented in this manuscript were developed in a dialogue with OpenAI's ChatGPT, primarily using GPT-5.6 Thinking. The model was first provided with the relevant literature concerning the origin of the ray in the one-dimensional Dirichlet spectrum and was asked to explore possible improvements of the known lower and upper bounds. It subsequently proposed the improvements developed in this paper and generated the corresponding proofs.

The model was then asked what further results could be obtained concerning isolated points of the Dirichlet spectrum. It proposed investigating whether all of Ivanov's points \(Z_k\) are isolated, determining asymptotic formulas for the lengths of the corresponding gaps, and describing the equivalence classes associated with each \(Z_k\). It subsequently generated proofs of these results.

The author carefully and independently checked all arguments, including the logical steps and calculations in the proofs. The Python code presented in the appendix was written independently by the author and was used to verify the computer-assisted part of the lower-bound argument. The author made the final decisions concerning the statements, proofs, organization, and presentation of the results and takes full responsibility for the content of the paper.

\end{document}